\title{On Symmetric Polynomials}
\author{Ryan Golden
        \and
        Ilwoo Cho\thanks{
                      Saint Ambrose Univ., Dept. of Math., 518 W. Locust St., Davenport, Iowa,
                      52803, USA
                  }\thanks{This paper is one of the results in the undergraduate research program of
                  St. Ambrose University, Department of Mathematics and Statistics. The
                  second-named author thanks the first-named author, for his hard works and
                  endless passion.}
        }
\documentclass{article}
\usepackage{graphicx}
\usepackage{amscd}
\usepackage[dvips]{epsfig}
\usepackage{amssymb}
\usepackage{amsmath}
\newtheorem{theorem}{Theorem}[section]
\newtheorem{corollary}[theorem]{Corollary}

\newtheorem{proposition}[theorem]{Proposition}

\newtheorem{example}{Example}[section]

\newtheorem{definition}{Definition}[section]
\newenvironment{proof}{{\sc Proof:}}{~\hfill QED}
\newenvironment{AMS}{}{}
\newenvironment{keywords}{}{}
\input tcilatex

\begin{document}
\newpage
\maketitle
\begin{abstract}
    In this paper, we study structure theorems of algebras of symmetric
    functions. Based on a certain relation on elementary symmetric polynomials
    generating such algebras, we consider perturbation in the algebras. In
    particular, we understand generators of the algebras as perturbations. From
    such perturbations, define injective maps on generators, which induce
    algebra-monomorphisms (or embeddings) on the algebras. They provide
    inductive structure theorems on algebras of symmetric polynomials. As
    application, we give a computer algorithm, written in JAVA v. 8, for finding
    quantities from elementary symmetric polynomials.
\end{abstract}



\section{Introduction}
In this paper, we study structure theorems of algebras generated by
symmetric polynomials with commutative multi-variables. By establishing
certain recurrence relations on symmetric polynomials, we prove our
structure theorems. As application, we consider how to construct a degree-$%
(n+1)$ single-variable polynomial $f_{t_{0}}(z)$ from a given degree-$n$
polynomial $f(z)$ by adding a zero $t_{0},$ and we provide a computer
algorithm, written in the computer language JAVA version 8; for finding
quantities obtained from elementary symmetric polynomials.

Throughout this paper, fix $n$  $\in$ $\mathbb{N},$ with additional condition: $n$ $>$ $1,$ and let $x_{1},$ ..., $x_{n}$ be arbitrary commutative \em{variables} (or \em{indeterminants}), for $n$ $\in $ $\mathbb{N}.$ Then one
can have an algebra

(1.1)

\begin{center}
	$\mathcal{A}_{x_{1},...,x_{n}}$ $=$ $\mathbb{C}\left[ \{x_{1},..., x_{n}\}\right] ,$
\end{center}

consisting of all $n$-variable polynomials in $x_{1},$ ..., $x_{n}.$ We call 
$\mathcal{A}_{x_{1},...,x_{n}}$ of (1.1), the $n$-\em{variable polynomial
	algebra}.

i.e., if $f$ $\in $ $\mathcal{A}_{x_{1},...,x_{n}},$ then it is expressed by

\begin{center}
	$f$ $=$ $f(x_{1},$ ..., $x_{n})$ $=$ $t_{0}+$ $\sum_{j=1}^{k}$ $\underset{(r_{1},...,r_{j})\in \{1,...,n\}^{j}}{\sum }$ $t_{(r_{1},...,r_{j})}$ $\left(\overset{j}{\underset{i=1}{\Pi}}x_{r_{i}}\right) ,$
\end{center}

with $t_{0},$ $t_{(r_{1},...,r_{j})}$ $\in $ $\mathbb{C},$ for $k$ $\in $ $\mathbb{N}.$

Let $X$ be a finite set. The \em{symmetric group }$S_{X}$\em{\ on} $X$
is a group under the usual functional composition consisting of all
bijective maps, called \em{permutations}, on $X.$ If $X$ $=$ $\{1,$ ..., $%
n\},$ then we denote $S_{X}$ simply by $S_{n},$ for $n$ $\in $ $\mathbb{N}.$

\begin{definition}
	An element $f$ of $\mathcal{A}_{x_{1},...,x_{n}}$ of (1.1) is said to be
	symmetric, if
	
	\begin{center}
		$f(x_{1},$ $x_{2},$ $...,$ $x_{n})$ $=$ $f\left( x_{\sigma (1)},\text{ }%
		x_{\sigma (2)},\text{ ..., }x_{\sigma (n)}\right) ,$
	\end{center}
	
	for all $\sigma $ $\in $ $S_{n},$ where $S_{n}$ is the symmetric group on $%
	\{1,$ ..., $n\}.$
	
	Define now the subset $\mathcal{S}_{x_{1},...,x_{n}}$ of $\mathcal{A}%
	_{x_{1},...,x_{n}}$ by
	
	(1.2)
	
	\begin{center}
		$\mathcal{S}_{x_{1},...,x_{n}}$ $=$ $\{f$ $\in $ $\mathcal{A}%
		_{x_{1},...,x_{n}}$ $:$ $f$ is symmetric$\}.$
	\end{center}
	
	We call $\mathcal{S}_{x_{1},...,x_{n}},$ the symmetric subalgebra of $%
	\mathcal{A}_{x_{1},...,x_{n}}.$
\end{definition}

\strut It is not difficult to check that,

\strut

(1.3)\qquad $f_{1},$ $f_{2}$ $\in $ $\mathcal{S}_{x_{1},...,x_{n}}$ $%
\Longrightarrow $ $f_{1}+f_{2}$ $\in $ $\mathcal{S}_{x_{1},...,x_{n}},$

(1.4) $\qquad f$ $\in $ $\mathcal{S}_{x_{1},...,x_{n}}$ and $t$ $\in $ $\mathbb{%
	C}$ $\Longrightarrow $ $tf$ $\in $ $\mathcal{S}_{x_{1},...,x_{n}},$

\strut

and hence, the subset $\mathcal{S}_{x_{1},...,x_{n}}$ of (1.2) forms a
well-defined vector subspace of $\mathcal{A}_{x_{1},...,x_{n}}$ over $\mathbb{C}%
,$ by (1.3) and (1.4).

Moreover, one has

\strut

(1.5) \qquad $f_{1},$ $f_{2}$ $\in $ $\mathcal{S}_{x_{1},...,x_{n}}$ $%
\Longrightarrow $ $f_{1}f_{2}$ $\in $ $\mathcal{S}_{x_{1},...,x_{n}},$

\strut

where $f_{1}f_{2}$ means the usual functional multiplication of $f_{1}$ and $f_{2}$ in $\mathcal{A}_{x_{1},...,x_{n}}.$

Since

\begin{center}
	$f_{1}\left( f_{2}f_{3}\right) =(f_{1}f_{2})f_{3}$ $\in $ $\mathcal{S}%
	_{x_{1},...,x_{n}},$
	
	\strut $f_{1}\left( f_{2}+f_{3}\right) $ $=$ $f_{1}f_{2}+f_{1}f_{3}$ $\in $ $%
	\mathcal{S}_{x_{1},...,x_{n}},$
\end{center}

and

\begin{center}
	$(f_{1}+f_{2})f_{3}$ $=$ $f_{1}f_{3}+f_{2}f_{3}$ $\in $ $\mathcal{S}%
	_{x_{1},...,x_{n}},$
\end{center}

by (1.5), whenever $f_{1},$ $f_{2},$ $f_{3}$ $\in $ $\mathcal{S}%
_{x_{1},...,x_{n}},$ the subspace $\mathcal{S}_{x_{1},...,x_{n}}$ indeed
forms a well-determined subalgebra of $\mathcal{A}_{x_{1},...,x_{n}}.$

\strut Let's define the following functions

(1.6)

\begin{center}
	$\varepsilon _{k}(x_{1},...,x_{n})$ $=$ $\underset{i_{1}<i_{2}<...<i_{k}\in
		\{1,...,n\}}{\sum }$ $\left( \overset{k}{\underset{l=1}{\Pi }}\text{ }%
	x_{i_{l}}\right) $
\end{center}

\strut

in $\mathcal{A}_{x_{1},...,x_{n}},$ \strut \strut for all $k$ $=$ $1,$ ..., $%
n.$ Then they are symmetric in $\mathcal{A}_{x_{1},...,x_{n}}.$ i.e.,

$\qquad \qquad \qquad \varepsilon _{1}(x_{1},...,x_{n})$ $=$ $\sum_{j=1}^{n}$
$x_{j},$

$\qquad \qquad \qquad \varepsilon _{2}(x_{1},...,x_{n})$ $=$ $\underset{%
	i_{1}<i_{2}\in \{1,...,n\}}{\sum }$ $x_{i_{1}}x_{i_{2}},$

$\qquad \qquad \qquad \varepsilon _{3}(x_{1},...,x_{n})$ $=$ $\underset{%
	i_{1}<i_{2}<i_{3}\in \{1,...,n\}}{\sum }$ $x_{i_{1}}x_{i_{2}}x_{i_{3}},$

$\cdot \cdot \cdot ,$

$\qquad \qquad \qquad \varepsilon _{n}(x_{1},$ ..., $x_{n})$ $=$ $\overset{n%
}{\underset{j=1}{\Pi }}x_{j}$

are elements of $\mathcal{S}_{x_{1},...,x_{n}}.$

\begin{definition}
	We call such polynomials $\varepsilon _{k}(x_{1},...,x_{n})$ of (1.6), the $%
	k $-\em{th elementary symmetric polynomials of }$\mathcal{S}%
	_{x_{1},...,x_{n}},$ for all $k$ $=$ $1,$ ..., $n.$
\end{definition}

\strut

\textbf{Notation 1.1} In the rest of this paper, we denote $\varepsilon
_{k}(x_{1},...,x_{n})$ simply by $\varepsilon _{k}^{1,...,n},$ for all $k$ $%
= $ $1,$ ..., $n.$ Also, for convenience, define

\begin{center}
	$\varepsilon _{0}^{1,...,n}$ $=$ $1$ and $\varepsilon _{n+i}^{1,...,n}$ $=$ $%
	0,$
\end{center}

as constant functions in $\mathcal{S}_{x_{1},...,x_{n}}$, for all $i$ $\in $ 
$\mathbb{N}.$ Whenever we want to emphasize the variables of $\varepsilon
_{k}^{1,...,n}$ precisely, we denote them by $\varepsilon
_{k}^{x_{1},...,x_{n}},$ for $k$ $=$ $1,$ ..., $n.$ $\square $

\strut

The following proposition is well-known under its name: \em{Fundamental
	Theorem of Symmetric Functions}.

\begin{proposition}
	(See [3]) Let $\mathcal{S}_{x_{1},...,x_{n}}$ be the symmetric subalgebra
	(1.2) of the $n$-variable polynomial algebra $\mathcal{A}_{x_{1},...,x_{n}}$
	of (1.1). Then
	
	(1.7)
	
	\begin{center}
		$\mathcal{S}_{x_{1},...,x_{n}}$ $\overset{\text{Alg}}{=}$ $\mathbb{C}\left[
		\{\varepsilon _{1}^{1,...,n},\text{ }\varepsilon _{2}^{1,...,n},\text{ ..., }%
		\varepsilon _{n}^{1,...,n}\}\right] ,$
	\end{center}
	
	where ``$\overset{\text{Alg}}{=}$'' means ``being algebra-isomorphic,'' and
	where $\varepsilon _{k}^{1,...,n}$ are the elementary symmetric polynomials
	in the sense of (1.6), for all $k$ $=$ $1,$ ..., n. $\square$ 
\end{proposition}

\em The above structure theorem (1.7) shows that all symmetric functions
in $\mathcal{A}_{x_{1},...,x_{n}}$ are generated by the elementary symmetric
polynomials $\{\varepsilon _{k}^{1,...,n}\}_{k=1}^{n}$ of (1.6).

For more about symmetric functions and related studies in mathematics, see
e.g., [2], [3], and cited papers therein.

\section{A Certain Relation on $\mathcal{S}_{x_{1},...,x_{n}}$}
	In this section, we establish a recurrence relation on the elementary
    symmetric polynomials $\{\varepsilon _{k}^{1,...,n}\}_{k=1}^{n}$ generating
    the symmetric subalgebra $\mathcal{S}_{x_{1},...,x_{n}}$ of the $n$-variable
    polynomial algebra $\mathcal{A}_{x_{1},...,x_{n}}.$
    
    \strut As in \textbf{Notation 1.1}, for any $k$ $\in $ $\{1,$ ..., $n\},$
    
    \begin{center}
    	$\varepsilon _{k}^{1,...,n}$ $=$ $\varepsilon _{k}(x_{1},...,x_{n}).$
    \end{center}
    
    So, if we write $\varepsilon _{k}^{i_{1},...,i_{t}},$ for $i_{j}$ $\in $ $%
    \{1,$ ..., $n\}$, $j$ $=$ $1,$ ..., $t,$ with $t$ $\leq $ $n$ in $\mathbb{N},$
    then it means
    
    (2.1)
    
    \begin{center}
    	$\varepsilon _{k}^{i_{1},...,i_{t}}$ $=$ $\left\{ 
    	\begin{array}{ll}
    	\varepsilon _{k}(x_{i_{1}},...x_{i_{t}}) & \text{if }k\leq t \\ 
    	0 & \text{if }k>t,
    	\end{array}
    	\right. $
    \end{center}
    
    Moreover, nonzero new elementary symmetric polynomials
    
    \begin{center}
    	$\varepsilon _{1}^{i_{1},...,i_{t}},$ ..., $\varepsilon
    	_{t}^{i_{1},...,i_{t}}$
    \end{center}
    
    are understood as the symmetric functions generating $\mathcal{S}%
    _{x_{i_{1}},...,x_{i_{t}}}$ in $\mathcal{A}_{x_{i_{1}},...,x_{i_{t}}}.$
    
    For example, if we write $\varepsilon _{k}^{1,...,n-1},$ for $k$ $=$ $1,$
    ..., $n-1,$ then they are the elementary symmetric polynomials generating $%
    \mathcal{S}_{x_{1},...,x_{n-1}}$ in $\mathcal{A}_{x_{1},...,x_{n-1}}$; if we
    write $\varepsilon _{1}^{t},$ then it is a single-variable function $x_{t}$
    in $\mathcal{A}_{x_{t}}$ $=$ $\mathbb{C}[x_{t}].$
    
    \begin{theorem}
    	Let $\varepsilon _{k}^{1,...,n}$ be the elementary symmetric polynomials
    	generating the symmetric subalgebra $\mathcal{S}_{x_{1},...,x_{n}}$. Then
    	
    	(2.2)
    	
    	\begin{center}
    		$\varepsilon _{k}^{1,...,n}$ $=$ $\varepsilon _{k}^{1,...,n-1}$ $+$ $%
    		\varepsilon _{k-1}^{1,...,n-1}$ $\varepsilon _{1}^{n},$
    	\end{center}
    	
    	where the summands and the factors of the right-hand side of (2.2) are in
    	the sense of (2.1).
    \end{theorem}
    
    \begin{proof}
    	Assume first that $k$ $=$ $1.$ Then
    	
    	$\qquad \qquad \varepsilon _{1}^{1,...,n}$ $=$ $\sum_{j=1}^{n}$ $x_{j}$ $=$ $%
    	\left( \sum_{j=1}^{n-1}x_{j}\right) $ $+$ $x_{n}$
    	
    	\strut
    	
    	$\qquad \qquad \qquad \qquad =$ $\varepsilon _{1}^{1,...,n-1}$ $+$ $1\cdot $ 
    	$\varepsilon _{1}^{n}$
    	
    	\strut
    	
    	$\qquad \qquad \qquad \qquad =$ $\varepsilon _{1}^{1,...,n-1}$ $+$ $%
    	\varepsilon _{0}^{1,...,n-1}$ $\varepsilon _{1}^{n},$
    	
    	\strut
    	
    	by \textbf{Notation 1.1} and (2.1). So, if $k$ $=$ $1,$ then the relation
    	(2.2) holds.
    	
    	Suppose now that $k$ $=$ $n.$ Then
    	
    	$\qquad \qquad \varepsilon _{n}^{1,...,n}$ $=$ $\overset{n}{\underset{j=1}{%
    			\Pi }}x_{j}$ $=$ $0+\left( \overset{n-1}{\underset{j=1}{\Pi }}\text{ }%
    	x_{j}\right) \left( x_{n}\right) $
    	
    	\strut
    	
    	$\qquad \qquad \qquad \qquad =$ $\varepsilon _{n}^{1,...,n-1}$ $+$ $%
    	\varepsilon _{n-1}^{1,...,n-1}$ $\varepsilon _{1}^{n},$
    	
    	\strut
    	
    	by \textbf{Notation 1.1} and (2.1). Thus, if $k$ $=$ $n,$ then the relation
    	(2.2) holds.
    	
    	Now, take $k$ $\in $ $\{2,$ ..., $n-1\},$ and define a set
    	
    	\begin{center}
    		$T_{k}^{1,...,n}$ $=$ $\{(i_{1},$ ..., $i_{k})$ $\in $ $\{1,$ ..., $n\}^{k}$ 
    		$:$ $i_{1}<i_{2}<...<i_{k}\}.$
    	\end{center}
    	
    	Define also a set
    	
    	\begin{center}
    		$T_{k}^{1,...,n-1}$ $=$ $\{(j_{1},...,j_{k})$ $\in $ $\{1,$ ..., $n-1\}^{k}$ 
    		$:$ $j_{1}<...<j_{k}\},$
    	\end{center}
    	
    	and similarly,
    	
    	\begin{center}
    		$T_{k-1}^{1,...,n-1}$ $=$ $\{(l_{1},...,l_{k-1})$ $\in $ $\{1,$ ..., $%
    		n-1\}^{k-1}$ $:$ $l_{1}$ $<$ ... $<$ $l_{k-1}\}.$
    	\end{center}
    	
    	By the very construction, two sets $T_{k}^{1,...,n-1}$ and $%
    	T_{k-1}^{1,...,n-1}$ are understood as subsets of $T_{k}^{1,...,n},$ for all 
    	$k$ $=$ $2,$ $3,$ ..., $n-1.$
    	
    	Depending on the above sets, construct
    	
    	\begin{center}
    		$X_{k}^{1,...,n}$ $=$ $\left\{ \overset{k}{\underset{l=1}{\Pi }}\text{ }%
    		x_{i_{l}}\text{ }:(i_{1},...,i_{k})\in T_{k}^{1,...,n}\right\} ,$
    	\end{center}
    	
    	\strut
    	
    	\begin{center}
    		$X_{k}^{1,...,n-1}$ $=$ $\left\{ \overset{k}{\underset{l=1}{\Pi }}\text{ }%
    		x_{i_{l}}:(i_{1},...,i_{k})\in T_{k}^{1,...,n-1}\right\} ,$
    	\end{center}
    	
    	and
    	
    	\begin{center}
    		$X_{k-1}^{1,...,n-1}$ $=$ $\left\{ \overset{k-1}{\underset{l=1}{\Pi }}\text{ 
    		}x_{i_{l}}:(i_{1},...,i_{k-1})\in T_{k-1}^{1,...,n-1}\right\} ,$
    	\end{center}
    	
    	respectively.
    	
    	Now, let's define
    	
    	\begin{center}
    		$Y_{k-1}^{1,...,n-1}$ $=$ $\left\{ \left( \overset{k-1}{\underset{l=1}{\Pi }}%
    		\text{ }x_{i_{l}}\right) x_{n}:\overset{k-1}{\underset{l=1}{\Pi }}\text{ }%
    		x_{i_{l}}\in X_{k-1}^{1,...,n-1}\right\} ,$
    	\end{center}
    	
    	i.e.,
    	
    	\begin{center}
    		$Y_{k-1}^{1,...,n-1}$ $=$ $X_{k-1}^{1,...,n-1}x_{n},$
    	\end{center}
    	
    	where $Xx$ $=$ $\{yx:y\in X\}.$ Notice that if $k$ $=$ $2,$ ..., $n-1,$ then
    	
    	(2.3)
    	
    	\begin{center}
    		$X_{k}^{1,...,n}$ $=$ $X_{k}^{1,...,n-1}$ $\sqcup $ $Y_{k-1}^{1,...,n-1},$
    	\end{center}
    	
    	set-theoretically, by the very definitions, where $\sqcup $ means the
    	disjoint union.
    	
    	Observe now that, whenever $k$ $=$ $2,$ ..., $n-1,$ we have
    	
    	$\qquad \varepsilon _{k}^{1,...,n}$ $=$ $\varepsilon _{k}\left(
    	x_{1},...,x_{n}\right) $ $=$ $\underset{(x_{i_{1}},...,x_{i_{k}})\in
    		X_{k}^{1,...,n}}{\sum }$ $\left( \overset{k}{\underset{l=1}{\Pi }}%
    	x_{i_{l}}\right) $
    	
    	\strut
    	
    	$\qquad \qquad =$ $\underset{(x_{i_{1}},...,x_{i_{k}})\in
    		X_{k}^{1,...,n-1}\sqcup Y_{k-1}^{1,...,n-1}}{\sum }\left( \overset{k}{%
    		\underset{l=1}{\Pi }}\text{ }x_{i_{l}}\right) $
    	
    	\strut by (2.3)
    	
    	$\qquad \qquad =$ $\underset{(x_{i_{1}},...,x_{i_{k}})\in X_{k}^{1,...,n-1}}{%
    		\sum }\left( \overset{k}{\underset{l=1}{\Pi }}\text{ }x_{i_{l}}\right) $
    	
    	$\qquad \qquad \qquad \qquad \qquad $ $+$ $\underset{%
    		(x_{i_{1}},...,x_{i_{k-1}},x_{n})\in Y_{k-1}^{1,...,n-1}}{\sum }$ $\left(
    	\left( \overset{k-1}{\underset{l=1}{\Pi }}\text{ }x_{i_{l}}\right)
    	(x_{n})\right) $
    	
    	\strut
    	
    	$\qquad \qquad =$ $\underset{(x_{i_{1}},...,x_{i_{k}})\in X_{k}^{1,...,n-1}}{%
    		\sum }\left( \overset{k}{\underset{l=1}{\Pi }}\text{ }x_{i_{l}}\right) $
    	
    	$\qquad \qquad \qquad \qquad \qquad +$ $\left( \underset{%
    		(x_{i_{1}},...,x_{i_{k-1}},x_{n})\in X_{k-1}^{1,...,n-1}}{\sum }\left( 
    	\overset{k-1}{\underset{l=1}{\Pi }}\text{ }x_{i_{l}}\right) \right) (x_{n})$
    	
    	\strut
    	
    	$\qquad \qquad =$ $\varepsilon _{k}^{1,...,n-1}$ $+$ $\varepsilon
    	_{k-1}^{1,...,n-1}$ $\varepsilon _{1}^{n}.$
    	
    	\strut
    	
    	Therefore, the relations (2.2) hold, for all $k$ $=$ $2,$ ..., $n-1.$\strut
    	
    	So, the relations (2.2) hold, for all $k$ $=$ $1,$ $2,$ ..., $n.$
    	\end{proof}
    
    The above relation (2.2) shows the relations between the generators
    
    \begin{center}
    	$\{\varepsilon _{1}^{1,...,n},$ ..., $\varepsilon _{n}^{1,...,n}\}$ of $%
    	\mathcal{S}_{x_{1},...,x_{n}}$
    \end{center}
    
    and those
    
    \begin{center}
    	$\{\varepsilon _{1}^{1,...,n-1},$ ..., $\varepsilon _{n-1}^{1,...,n-1}\}$ of 
    	$\mathcal{S}_{x_{1},...,x_{n-1}}$ and $\{\varepsilon _{1}^{n}\}$ of $%
    	\mathcal{S}_{x_{n}}$ $=$ $\mathcal{A}_{x_{n}}$
    \end{center}
    
    \strut Also, it shows that all generators of $\mathcal{S}_{x_{1},...,x_{n}}$
    are induced by the generators of
    
    \begin{center}
    	$\mathcal{S}_{x_{1},...,x_{n-1}}$ and $\mathcal{S}_{x_{n}},$
    \end{center}
    
    by (2.1).
    
    Motivated by (2.2), one can obtain the following generalized result.
    
    \begin{theorem}
    	Let $\{\varepsilon _{k}^{1,...,n}\}_{k=1}^{n}$ be the elementary symmetric
    	polynomials generating the symmetric subalgebra $\mathcal{S}%
    	_{x_{1},...,x_{n}}$. Then, for any fixed $i_{0}$ $\in $ $\{1,$ ..., $n\},$
    	we have
    	
    	(2.4)
    	
    	\begin{center}
    		$\varepsilon _{k}^{1,...,n}$ $=$ $\varepsilon _{k}^{1,...,i_{0}-1,\text{ }%
    			i_{0}+1,...,n}$ $+$ $\varepsilon _{k-1}^{1,...,i_{0}-1,i_{0}+1,...,n}$ $%
    		\varepsilon _{1}^{i_{0}},$
    	\end{center}
    	
    	for all $k$ $=$ $1,$ $2,$ ..., $n,$ where $\varepsilon
    	_{l}^{i_{1},...,i_{m}} $ are in the sense of (2.1).
    \end{theorem}
    
    \begin{proof}
    	\strut The proof of (2.4) is similarly done by that of (2.3), by replacing $%
    	n $ to $i_{0}.$
    \end{proof}
    
    \strut So, one can realize that the generators of symmetric subalgebra $%
    \mathcal{S}_{x_{1},...,x_{n}}$ are induced by the generators of
    
    \begin{center}
    	$\mathcal{S}_{x_{1},...,x_{i_{0}-1},x_{i_{0}+1},...,n}$ and $\mathcal{S}%
    	_{x_{i_{0}}}$ $=$ $\mathbb{C}[x_{i_{0}}=\varepsilon _{1}^{i_{0}}],$
    \end{center}
    
    by (2.4).
    
    \begin{example}
    	Suppose we have the symmetric subalgebra $\mathcal{S}%
    	_{x_{1},x_{2},x_{3},x_{4}}$ whose generators are the elementary symmetric
    	polynomials $\varepsilon _{j}^{1,...,4},$ for $j$ $=$ $1,$ ..., $4.$ i.e.,
    	
    	$\qquad \qquad \varepsilon _{1}^{1,2,3,4}$ $=$ $x_{1}+x_{2}+x_{3}+x_{4},$
    	
    	$\qquad \qquad \varepsilon _{2}^{1,2,3,4}$ $=$ $%
    	x_{1}x_{2}+x_{1}x_{3}+x_{1}x_{4}+x_{2}x_{3}+x_{2}x_{4}+x_{3}x_{4},$
    	
    	$\qquad \qquad \varepsilon _{3}^{1,2,3,4}$ $=$ $%
    	x_{1}x_{2}x_{3}+x_{1}x_{3}x_{4}+x_{2}x_{3}x_{4},$
    	
    	and
    	
    	$\qquad \qquad \varepsilon _{4}^{1,2,3,4}$ $=$ $x_{1}x_{2}x_{3}x_{4}$
    	
    	in $\mathcal{S}_{x_{1},x_{2},x_{3},x_{4}}.$
    	
    	Then
    	
    	$\qquad \qquad \varepsilon _{1}^{1,2,3,4}$ $=$ $\varepsilon
    	_{1}^{1,2,3}+\varepsilon _{0}^{1,2,3}\varepsilon _{1}^{4}$ $=$ $\varepsilon
    	_{1}^{1,2,3}$ $+$ $\varepsilon _{1}^{4},$
    	
    	$\qquad \qquad \varepsilon _{2}^{1,2,3,4}$ $=$ $\varepsilon
    	_{2}^{1,2,3}+\varepsilon _{1}^{1,2,3}\varepsilon _{1}^{4},$
    	
    	$\qquad \qquad \varepsilon _{3}^{1,2,3,4}$ $=$ $\varepsilon
    	_{3}^{1,2,3}+\varepsilon _{2}^{1,2,3}\varepsilon _{1}^{4},$
    	
    	and
    	
    	$\qquad \qquad \varepsilon _{4}^{1,2,3,4}$ $=$ $\varepsilon
    	_{4}^{1,2,3}+\varepsilon _{3}^{1,2,3}\varepsilon _{1}^{4}$ $=$ $\varepsilon
    	_{3}^{1,2,3}\varepsilon _{1}^{4}.$
    	
    	\strut
    	
    	Similarly, one has
    	
    	\strut
    	
    	$\qquad \qquad \varepsilon _{1}^{1,2,3,4}$ $=$ $\varepsilon
    	_{1}^{1,2,4}+\varepsilon _{0}^{1,2,4}\varepsilon _{1}^{3}$ $=$ $\varepsilon
    	_{1}^{1,2,4}+\varepsilon _{1}^{3},$
    	
    	$\qquad \qquad \varepsilon _{2}^{1,2,3,4}$ $=$ $\varepsilon
    	_{2}^{1,2,4}+\varepsilon _{1}^{1,2,4}\varepsilon _{1}^{3},$
    	
    	$\qquad \qquad \varepsilon _{3}^{1,2,3,4}$ $=$ $\varepsilon
    	_{3}^{1,2,4}+\varepsilon _{2}^{1,2,4}\varepsilon _{1}^{3},$
    	
    	and
    	
    	$\qquad \qquad \varepsilon _{4}^{1,2,3,4}$ $=$ $\varepsilon
    	_{4}^{1,2,4}+\varepsilon _{3}^{1,2,4}\varepsilon _{1}^{3}$ $=$ $\varepsilon
    	_{3}^{1,2,4}\varepsilon _{1}^{3},$
    	
    	etc.
    \end{example}
    
    \strut \strut
    
    \textbf{Acknowledgment} Before submitting this paper, the authors realized
    that the relation (2.2) was already known in psychology; psychological tests
    (e.g, see [1]). According to [1], Dr. Fischer proved the relation (2.2) in
    his book ''Einf\"{u}hrung in die Theorie Psychologischer Tests.''
    However, we could not find the sources, including the book, containing
    Fischer's proof. So, we provided our own proofs of (2.2) and (2.4) above. $%
    \square $
    
    \strut \strut \strut \strut
    
\section{Perturbations on $\mathfrak{S}$ and Shifts on $\mathcal{S}_{x_{1},...,x_{n}}$}
    
In this section, we define the collection $\mathfrak{S}$ of all symmetric
subalgebras of finitely-many commutative variables. And construct certain
perturbation processes on $\mathfrak{S},$ by understanding each $f$ of $\mathcal{S}_{x_{1},...,x_{n}}$ $\in $ $\mathfrak{S}$ as multiplication from $\mathcal{S}%
_{y_{1},...,y_{k}}$ to $\mathcal{S}_{x_{1},...,x_{n},y_{1},...,y_{k}}$ in $%
\mathfrak{S}$ under additional axiomatization (See Section 3.1). Also, we
consider a shifting process on a fixed symmetric subalgebra $\mathcal{S}%
_{x_{1},...,x_{n}}$ $\in $ $\mathfrak{S}$ by reformulating indexes of
generators of $\mathcal{S}_{x_{1},...,x_{n}}$, in Section 3.2. In Section
3.3, we apply our perturbation of Section 3.1 and shifting process of
Section 3.2 to the inductive construction processes of $\mathfrak{S}.$
    
\subsection{Perturbation on $\mathfrak{S}$}
Let $\mathfrak{S}$ be the collection of symmetric subalgebras $\mathcal{S}%
_{x_{1},...,x_{n}}$ in commutative variables $\{x_{1},$ ..., $x_{n}\},$ for
all $n$ $\in $ $\mathbb{N}.$ Let's fix $n_{0}$ $\in $ $\mathbb{N},$ and take $%
\mathcal{S}_{x_{1},...,x_{n_{0}}}$ in $\mathfrak{S}.$

\begin{definition}
Define now \em{perturbations of }$\mathcal{S}_{x_{1},...,x_{n_{0}}}$\em{%
	\ on} $\mathfrak{S}$ by

(3.1.1)

\begin{center}
$f$ $:$ $h$ $\in $ $\mathcal{S}_{y_{1},...,y_{n}}$ $\longmapsto $ $fh$ $\in $
$\mathcal{S}_{x_{1},...,x_{n_{0}},y_{1},...,y_{n}}$ in $\mathfrak{S},$
\end{center}

for $f$ $\in $ $\mathcal{S}_{x_{1},..,x_{n_{0}}},$ for all $\mathcal{S}%
_{y_{1},...,y_{n}}$ $\in $ $\mathfrak{S},$ with identification: if

\begin{center}
$\partial $ $=$ $\{x_{1},...,x_{n_{0}}\}$ $\cap $ $\{y_{1},...,y_{n}\}$
\end{center}

is non-empty, then

$\strut $(3.1.1)$^{\prime }$

\begin{center}
$\mathcal{S}_{x_{1},...,x_{n_{0}},y_{1},...,y_{n}}$ $=$ $\mathcal{S}_{\left(
	\{x_{1},..,x_{n_{0}}\}\setminus \partial \right) \sqcup \text{ }\partial 
	\text{ }\sqcup \text{ }\left( \{y_{1},...,y_{n}\}\setminus \partial \right)
},$
\end{center}

in $\mathfrak{S}.$
\end{definition}

Since every symmetric subalgebra is generated by elementary symmetric
polynomials, the perturbations (3.1.1) of $\mathcal{S}_{x_{1},...,x_{n_{0}}}$
on $\mathfrak{S}$ is characterized on generators, i.e.,

(3.1.2)

\begin{center}
$\varepsilon _{k}^{x_{1},...,x_{n_{0}}}$ $:$ $\varepsilon
_{l}^{y_{1},...,y_{n}}$ $\in $ $\mathcal{S}_{y_{1},...,y_{n}}$ $\longmapsto $
$\varepsilon _{k}^{x_{1},...,x_{n_{0}}}\varepsilon _{l}^{y_{1},...,y_{n}}$ $%
\in $ $\mathcal{S}_{x_{1},..,x_{n_{0}},y_{1},...,y_{n}}$
\end{center}

in $\mathfrak{S},$ satisfying the identification; (3.1.1)$^{\prime }.$

Now, let $\mathcal{S}_{y}$ $=$ $\mathbb{C}[y]$ $=$ $\mathcal{A}_{y}$ in $\mathfrak{S%
},$ and assume that $y$ $\neq $ $x_{j},$ for $j$ $=$ $1,$ ..., $n_{0}.$
Consider the perturbation of $\mathcal{S}_{x_{1},...,x_{n_{0}}}$ on $\mathfrak{S}
$ acting on $\mathcal{S}_{y},$ i.e.,

\begin{center}
$\varepsilon _{k}^{x_{1},...,x_{n_{0}}}$ $:$ $\varepsilon _{1}^{y}$ $=$ $y$ $%
\longmapsto $ $\varepsilon _{k}^{x_{1},...,x_{n_{0}}}$ $\varepsilon
_{1}^{y}. $
\end{center}

\begin{proposition}
The perturbations of $\mathcal{S}_{x_{1},...,x_{n_{0}}}$ on $\mathfrak{S}$ is a
well-defined categorial functor on $\mathfrak{S}.$
\end{proposition}

\begin{proof}
The proof is from the very definition (3.1.1), with identification (3.1.1)$%
^{\prime },$ with help of (3.1.2).\strut
\end{proof}
    
\subsection{\strut Shifts on $\mathcal{S}_{x_{1},...,x_{n}}$}
Let $\mathcal{S}_{x_{1},...,x_{n}}$ $\in $ $\mathfrak{S}$ be a symmetric
subalgebra. Define a shift $U$ on $\mathcal{S}_{x_{1},...,x_{n}}$ by a
linear multiplicative transformation on $\mathcal{S}_{x_{1},...,x_{n}}$
satisfying

(3.2.1)

\begin{center}
$U$ $:$ $\varepsilon _{k}^{1,...,n}$ $\longmapsto $ $\varepsilon
_{k-1}^{1,...,n}$ on $\mathcal{S}_{x_{1},...,x_{n}},$
\end{center}

for all $k$ $=$ $1,$ ..., $n,$ with additional axiomatization;

(3.2.1)$^{\prime }$

\begin{center}
\strut $U\left( \varepsilon _{0}^{1,...,n}\right) $ $=$ $U\left( 1\right) $ $%
=$ $0,$
\end{center}

making all constant functions of $\mathbb{C}$ in $\mathcal{S}_{x_{1},...,x_{n}}$
be zero, i.e., $U\left( C\right) $ $=$ $0,$ for all $C$ $\in $ $\mathbb{C}$ in $%
\mathcal{S}_{x_{1},...,x_{n}}.$\strut

\begin{definition}
The morphism $U$ on $\mathcal{S}_{x_{1},...,x_{n}}$ of (3.2.1) is called the
shift on $\mathcal{S}_{x_{1},...,x_{n}}.$
\end{definition}

More generally, the shift $U$ on $\mathcal{S}_{x_{1},...,x_{n}}$ satisfies

$\qquad \qquad U\left( t_{0}+\sum_{t=0}^{k}\underset{(i_{1},...,i_{k})\in
	\{1,...,n\}^{k}}{\sum }t_{i_{1},..,i_{k}}\text{ }\overset{k}{\underset{j=1}{%
		\Pi }}\varepsilon _{i_{j}}^{1,...,n}\right) $

(3.2.2)

$\qquad \qquad \qquad \qquad =$ $\sum_{t=0}^{k}\underset{(i_{1},...,i_{k})%
	\in \{1,...,n\}^{k}}{\sum }t_{i_{1},..,i_{k}}$ $\overset{k}{\underset{j=1}{%
		\Pi }}\varepsilon _{i_{j}-1}^{1,...,n},$

under (3.2.1)$^{\prime },$ where $t_{i_{1},...,i_{k}}$ $\in $ $\mathbb{C}.$

\begin{proposition}
The shift $U$ of (3.2.1), satisfying (3.2.2), is a well-defined
algebra-homomorphism on $\mathcal{S}_{x_{1},...,x_{n}}$, for all $\mathcal{S}%
_{x_{1},...,x_{n}}$ $\in $ $\mathfrak{S}.$
\end{proposition}

\begin{proof}
Let $U$ be the shift (3.1.1) on $\mathcal{S}_{x_{1},...,x_{n}}$ $\in $ $%
\mathfrak{S}$ satisfying (3.1.1)$^{\prime }.$ By the very construction, the
morphism $U$ is a linear transformation which is multiplicative. So, it is
automatically an algebra-homomorphism.

Clearly, by the linearity, one has

\begin{center}
$U\left( t_{1}f_{1}+t_{2}f_{2}\right) $ $=$ $t_{1}U\left( f_{1}\right)
+t_{2}U\left( f_{2}\right) ,$
\end{center}

for all $t_{1},$ $t_{2}$ $\in $ $\mathbb{C}$ and $f_{1},$ $f_{2}$ $\in $ $%
\mathcal{S}_{x_{1},...,x_{n}}.$

Also, by the multiplicativity of $U,$

$\qquad U\left( \left( \varepsilon _{k_{1}}^{1,...,n}\right)
^{m_{1}}...\left( \varepsilon _{k_{l}}^{1,...,n}\right) ^{m_{l}}\right) $ $=$
$U\left( \left( \varepsilon _{k_{1}}^{1,...,n}\right) ^{m_{1}}\right)
...U\left( \left( \varepsilon _{k_{l}}^{1,...,n}\right) ^{m_{l}}\right) $

\strut

$\qquad \qquad \qquad =$ $\left( U\left( \varepsilon
_{k_{1}}^{1,...,n}\right) \right) ^{m_{1}}...\left( U\left( \varepsilon
_{k_{l}}^{1,...,n}\right) \right) ^{m_{l}}$

\strut

$\qquad \qquad \qquad =$ $\left( \varepsilon _{k_{1}-1}^{1,...,n}\right)
^{m_{1}}...\left( \varepsilon _{k_{l}}^{1,...,n}\right) ^{m_{l}},$

by (3.2.2), for all $k_{1},$ ..., $k_{l}$ $=$ $0,$ $1,$ ..., $n$ (with
(3.2.1)$^{\prime }$), for all $m_{1},$ ..., $m_{l}$ $\in $ $\mathbb{N},$ for
all $l$ $\in $ $\mathbb{N}.$

So, for any $f_{1},$ $f_{2}$ $\in $ $\mathcal{S}_{x_{1},...,x_{n}},$ we have

\begin{center}
$U\left( f_{1}f_{2}\right) $ $=$ $U\left( f_{1}\right) $ $U\left(
f_{2}\right) $ in $\mathcal{S}_{x_{1},...,x_{n}}.$
\end{center}

i.e., the morphism $U$ is indeed an algebra-homomorphism on $\mathcal{S}%
_{x_{1},...,x_{n}}.$
\end{proof}

\subsection{\strut From $\mathcal{S}_{x_{1},...,x_{n}}$ to $\mathcal{S}_{x_{1},...,x_{n},y}$ in $\mathfrak{S}$}
Now, let's fix $n_{0}$ $\in $ $\mathbb{N},$ and the symmetric subalgebra $%
\mathcal{S}_{x_{1},...,x_{n_{0}}}$ in $\mathfrak{S}.$ Also, fix a symmetric
subalgebra $\mathcal{S}_{y}$ $=$ $\mathbb{C}[y]$ $=$ $\mathcal{A}_{y}$ in $%
\mathfrak{S},$ where $y$ $\neq $ $x_{j},$ for all $j$ $=$ $1,$ ..., $n_{0}.$ Let

\begin{center}
	$\mathcal{E}_{x_{1},...,x_{n_{0}}}$ $=$ $\{\varepsilon
	_{k}^{x_{1},...,x_{n_{0}}}\}_{k=1}^{n_{0}}$
\end{center}

be the generator set of $\mathcal{S}_{x_{1},...,x_{n_{0}}}$. If we
understand generators of $\mathcal{E}_{x_{1},...,x_{n_{0}}}$ as the
perturbations (3.1.1) on $\mathfrak{S},$ then they act

(3.3.1)

\begin{center}
	$\varepsilon _{k}^{x_{1},...,x_{n_{0}}}\left( \varepsilon _{1}^{y}\right) $ $%
	=$ $\varepsilon _{k}^{x_{1},...,x_{n_{0}}}\varepsilon _{1}^{y}$ $\in $ $%
	\mathcal{S}_{x_{1},...,x_{n_{0}},y}$
\end{center}

on $\mathcal{S}_{y}$, for all $k$ $=$ $0,$ $1,$ ..., $n_{0},$ with identity: 
$\varepsilon _{0}^{x_{1},...,x_{n_{0}}}$ $=$ $1.$

On the perturbations $\mathcal{E}_{x_{1},...,x_{n_{0}}}$ of (3.3.1), consider

(3.3.2)

\begin{center}
	$U\left( \varepsilon _{k}^{x_{1},...,x_{n_{0}}}\right) $ $=$ $\varepsilon
	_{k-1}^{x_{1},...,x_{n_{0}}},$ for all $k$ $=$ $1,$ ..., $n_{0},$
\end{center}

where $U$ is the shift on $\mathcal{S}_{x_{1},...,x_{n_{0}}}$ of (3.2.1)
satisfying (3.2.1)$^{\prime }.$

Define now a morphism

\begin{center}
	$\alpha $ $:$ $\mathcal{E}_{x_{1},...,x_{n_{0}}}$ $\rightarrow $ $\mathcal{S}%
	_{x_{1},...,x_{n_{0}},y},$
\end{center}

by a function satisfying

(3.3.3)

\begin{center}
	$\alpha \left( \varepsilon _{k}^{x_{1},...,x_{n_{0}}}\right) $ $=$ $%
	\varepsilon _{k}^{x_{1},...,x_{n_{0}}}+U\left( \varepsilon
	_{k}^{x_{1},...,x_{n_{0}}}\right) \varepsilon _{1}^{y}.$
\end{center}

\begin{theorem}
	The function $\alpha $ of (3.3.3) is a well-defined injective function from $%
	\mathcal{E}_{x_{1},...,x_{n_{0}}}$ into $\mathcal{S}%
	_{x_{1},...,x_{n_{0}},y}. $ Furthermore, this function $\alpha $ of (3.3.3)
	is injective from the generator set $\mathcal{E}_{x_{1},...,x_{n_{0}}}$ of $%
	\mathcal{S}_{x_{1},...,x_{n_{0}}}$ into the generator set $\mathcal{E}%
	_{x_{1},...,x_{n_{0}},y}$ of $\mathcal{S}_{x_{1},...,x_{n_{0}},y}.$ In
	particular, one has
	
	(3.3.4)
	
	\begin{center}
		$\mathcal{E}_{x_{1},...,x_{n_{0}},y}$ $=$ $\alpha \left( \mathcal{E}%
		_{x_{1},...,x_{n_{0}}}\right) $ $\sqcup $ $\{\varepsilon
		_{n_{0}}^{x_{1},...,x_{n_{0}}}\varepsilon _{1}^{y}\}$.
	\end{center}
\end{theorem}

\begin{proof}
	\strut By the very definition (3.3.3), the function $\alpha $ has its domain 
	$\mathcal{E}_{x_{1},...,x_{n_{0}}},$ whose range is contained in $\mathcal{S}%
	_{x_{1},...,x_{n_{0}},y}.$ It is not difficult to check that $\alpha $ is
	injective. Indeed, whenever $k_{1}$ $\neq $ $k_{2}$ in $\{1,$ ..., $n_{0}\},$
	
	\begin{center}
		$
		\begin{array}{ll}
		\alpha \left( \varepsilon _{k_{1}}^{x_{1},...,x_{n_{0}}}\right) & 
		=\varepsilon _{k_{1}}^{x_{1},...,x_{n_{0}}}+\varepsilon
		_{k_{1}-1}^{x_{1},...,x_{n_{0}}}\varepsilon _{1}^{y} \\ 
		&  \\ 
		& \neq \varepsilon _{k_{2}}^{x_{1},...,x_{n_{0}}}+\varepsilon
		_{k_{2}}^{x_{1},...,x_{n_{0}}}\varepsilon _{1}^{y}=\alpha \left( \varepsilon
		_{k_{2}}^{x_{1},...,x_{n_{0}}}\right) ,
		\end{array}
		$
	\end{center}
	
	in $\mathcal{S}_{x_{1},...,x_{n_{0}},y},$ by (1.6).
	
	Again by (3.3.3), the range $\alpha \left( \mathcal{E}_{x_{1},...,x_{n_{0}}}%
	\right) $ of this map $\alpha $ is contained in the generator set $\mathcal{E%
	}_{x_{1},...,x_{n_{0}},y}$ of $\mathcal{S}_{x_{1},...,x_{n_{0}},y}.$ Indeed,
	the symmetric subalgebra $\mathcal{S}_{x_{1},...,x_{n_{0}},y}$ is generated
	by the elementary symmetric polynomials,
	
	\begin{center}
		$\varepsilon _{l}^{x_{1},...,x_{n_{0}},y}$ $=$ $\varepsilon _{l}\left( x_{1},%
		\text{ ..., }x_{n_{0}},\text{ }y\right) ,$
	\end{center}
	
	for $l$ $=$ $1,$ $2,$ ..., $n_{0}+1,$ satisfying
	
	\begin{center}
		$
		\begin{array}{ll}
		\varepsilon _{l}^{x_{1},...,x_{n_{0}},y} & =\varepsilon
		_{l}^{x_{1},...,x_{n_{0}}}+\varepsilon
		_{l-1}^{x_{1},...,x_{n_{0}}}\varepsilon _{1}^{y} \\ 
		& =\varepsilon _{l}^{x_{1},...,x_{n_{0}}}+U\left( \varepsilon
		_{l}^{x_{1},...,x_{n_{0}}}\right) \varepsilon _{1}^{y}=\alpha \left(
		\varepsilon _{l}^{x_{1},...,x_{n_{0}}}\right) ,
		\end{array}
		$
	\end{center}
	
	by (2.2), for all $l$ $=$ $1,$ $2,$ ..., $n_{0}.$ Therefore,
	
	\begin{center}
		$\mathcal{E}_{x_{1},...,x_{n_{0}},y}$ $=$ $\alpha \left( \mathcal{E}%
		_{x_{1},...,x_{n_{0}}}\right) $ $\sqcup $ $\{\varepsilon
		_{n_{0}+1}^{x_{1},...,x_{n_{0}},y}$ $=$ $\varepsilon
		_{n_{0}}^{x_{1},...,x_{n_{0}}}\varepsilon _{1}^{y}\}.$
	\end{center}
	
	\strut
\end{proof}

\strut The above theorem with the relation (3.3.4) illustrates the embedding
property of the generator set $\mathcal{E}_{x_{1},...,x_{n_{0}}}$ of $%
\mathcal{S}_{x_{1},...,x_{n_{0}}}$ into the generator set $\mathcal{E}%
_{x_{1},...,x_{n_{0}},y}$ of $\mathcal{S}_{x_{1},...,x_{n_{0}},y}.$

Define now a linear multiplicative morphism

\begin{center}
	$\Phi $ $:$ $\mathcal{S}_{x_{1},...,x_{n_{0}}}$ $\rightarrow $ $\mathcal{S}%
	_{x_{1},...,x_{n_{0}},y}$
\end{center}

by

$\qquad \qquad \Phi \left( t_{0}+\sum_{j=1}^{k}\underset{(i_{1},...,i_{k})%
	\in \{1,...,n_{0}\}^{k}}{\sum }t_{i_{1},...,i_{k}}\text{ }\overset{k}{%
	\underset{l=1}{\Pi }}\varepsilon _{i_{l}}^{x_{1},...,x_{n_{0}}}\right) $

\strut (3.3.5)

$\qquad \qquad \qquad \qquad =$ $t_{0}+\sum_{j=1}^{k}\underset{%
	(i_{1},...,i_{k})\in \{1,...,n_{0}\}^{k}}{\sum }t_{i_{1},...,i_{k}}$ $%
\overset{k}{\underset{l=1}{\Pi }}\alpha \left( \varepsilon
_{i_{l}}^{x_{1},...,x_{n_{0}}}\right) ,$

\strut \strut

for all $k$ $\in $ $\mathbb{N},$ where $t_{0},$ $t_{i_{1},...,i_{k}}$ $\in $ $%
\mathbb{C}.$

The above linear multiplicative morphism $\Phi $ of (3.3.5) is well-defined
because the function $\alpha $ of (3.3.3) is well-defined, and it preserves
the generators $\mathcal{E}_{x_{1},...,x_{n_{0}}}$ injectively, by (3.3.4),
into $\mathcal{E}_{x_{1},...,x_{n_{0}},y}$ of $\mathcal{S}%
_{x_{1},...,x_{n_{0}},y}.$

\begin{corollary}
	The symmetric subalgebra $\mathcal{S}_{x_{1},...,x_{n_{0}}}$ is
	algebra-monomorphic to the symmetric subalgebra $\mathcal{S}%
	_{x_{1},...,x_{n_{0}},y},$ for fixed variables $x_{1},$ ..., $x_{n_{0}},$
	i.e.,
	
	(3.3.6)
	
	\begin{center}
		$\mathcal{S}_{x_{1},...,x_{n_{0}}}$ $\overset{\text{Alg}}{\hookrightarrow }$ 
		$\mathcal{S}_{x_{1},...,x_{n_{0}},y},$
	\end{center}
	
	where ``$\overset{\text{Alg}}{\hookrightarrow }$'' means ``being embedded
	in.''
\end{corollary}

\begin{proof}
	By the algebra-monomorphism $\Phi $ of (3.3.5), $\mathcal{S}%
	_{x_{1},...,x_{n_{0}}}$ is algebra-monomorphic to $\mathcal{S}%
	_{x_{1},...,x_{n_{0}},y},$ equivalently, $\mathcal{S}_{x_{1},...,x_{n_{0}}}$
	is naturally embedded in $\mathcal{S}_{x_{1},...,x_{n_{0}},y}.$
\end{proof}

More precise to (3.3.6), we obtain the following structure theorem.

\begin{theorem}
	Let $\Phi $ be the algebra-monomorphism (or the embedding) (3.3.5) of $%
	\mathcal{S}_{x_{1},...,x_{n_{0}}}$ into $\mathcal{S}%
	_{x_{1},...,x_{n_{0}},y}. $ Then
	
	(3.3.7)
	
	\begin{center}
		$\mathcal{S}_{x_{1},...,x_{n_{0}},y}$ $\overset{\text{Alg}}{=}$ $\Phi \left( 
		\mathcal{S}_{x_{1},...,x_{n_{0}}}\right) \oplus \mathbb{C}\left[ \{\varepsilon
		_{n_{0}}^{x_{1},...,x_{n_{0}}}\varepsilon _{1}^{y}\}\right] ,$
	\end{center}
	
	where $\mathbb{C}[X]$ mean the algebras generated by sets $X,$ and $\oplus $
	means the (pure-algebraic) direct product on algebras.
\end{theorem}

\begin{proof}
	\strut Note that
	
	$\qquad \qquad \mathcal{S}_{x_{1},...,x_{n_{0}},y}$ $=$ $\mathbb{C}\left[ 
	\mathcal{E}_{x_{1},...,x_{n_{0}},y}\right] $
	
	by (1.7)
	
	$\qquad \qquad \qquad =$ $\mathbb{C}\left[ \alpha \left( \mathcal{E}%
	_{x_{1},...,x_{n_{0}}}\right) \sqcup \{\varepsilon
	_{n_{0}}^{x_{1},...,x_{n_{0}}}\varepsilon _{1}^{y}\}\right] $
	
	by (3.3.4)
	
	$\qquad \qquad \qquad =$ $\mathbb{C}\left[ \mathbb{C}\left[ \alpha \left( \mathcal{%
		E}_{x_{1},...,x_{n_{0}}}\right) \right] \sqcup \mathbb{C}\left[ \varepsilon
	_{n_{0}}^{x_{1},...,x_{n_{0}}}\varepsilon _{1}^{y}\right] \right] $
	
	\strut by construction
	
	$\qquad \qquad \qquad =$ $\mathbb{C}\left[ \Phi \left( \mathcal{S}%
	_{x_{1},...,x_{n_{0}}}\right) \right] \oplus \mathbb{C}\left[ \varepsilon
	_{n_{0}}^{x_{1},...,x_{n_{0}}}\varepsilon _{1}^{y}\right] $
	
	by (3.3.5)
	
	$\qquad \qquad \qquad =$ $\Phi \left( \mathcal{S}_{x_{1},...,x_{n_{0}}}%
	\right) \oplus \mathbb{C}\left[ \varepsilon
	_{n_{0}}^{x_{1},...,x_{n_{0}}}\varepsilon _{1}^{y}\right] ,$
	
	since $\Phi $ is an embedding. Therefore, the isomorphism theorem (3.3.7) is
	obtained.
\end{proof}

\strut By (3.3.7) and (2.4), we obtain the following theorem, too.

\begin{theorem}
	Let $n$ $\in $ $\mathbb{N},$ and let $\mathcal{S}_{x_{1},...,x_{n+1}}$ be the
	symmetric algebra in $\{x_{1},$ $...,$ $x_{n},$ $x_{n+1}\}.$ Then
	
	(3.3.8)
	
	\begin{center}
		$\mathcal{S}_{x_{1},...,x_{n+1}}$ $\overset{\text{Alg}}{=}$ $\Phi \left( 
		\mathcal{S}_{x_{1},...,x_{j-1},x_{j+1},...,x_{n+1}}\right) $ $\oplus $ $\mathbb{%
			C}\left[ \varepsilon _{n}^{x_{1},...,x_{j-1},x_{j+1},...,x_{n}}\varepsilon
		_{1}^{x_{j}}\right] ,$
	\end{center}
	
	where $\Phi $ is in the sense of (3.3.5).
\end{theorem}

\begin{proof}
	\strut The proof of the structure theorem (3.3.8) is done by that of (3.3.7)
	in terms of (2.4).
\end{proof}

\strut We finish this section with the following example.

\begin{example}
	Let $\mathcal{S}_{x_{1},x_{2},x_{3}}$ be the symmetric subalgebra in $%
	\{x_{1},x_{2},x_{3}\},$ with its generator set
	
	\begin{center}
		$\mathcal{E}_{x_{1},x_{2},x_{3}}$ $=$ $\{\varepsilon
		_{1}^{x_{1},x_{2},x_{3}},$ $\varepsilon _{2}^{x_{1},x_{2},x_{3}},$ $%
		\varepsilon _{3}^{x_{1},x_{2},x_{3}}\},$
	\end{center}
	
	where
	
	$\qquad \qquad \varepsilon _{1}^{x_{1},x_{2},x_{3}}$ $=$ $x_{1}+x_{2}+x_{3},$
	
	$\qquad \qquad \varepsilon _{2}^{x_{1},x_{2},x_{3}}$ $=$ $%
	x_{1}x_{2}+x_{1}x_{3}+x_{2}x_{3},$
	
	and
	
	$\qquad \qquad \varepsilon _{3}^{x_{1},x_{2},x_{3}}$ $=$ $x_{1}x_{2}x_{3}.$
	
	For the injective map $\alpha $ of (3.3.3), we have
	
	\begin{center}
		$\alpha (\varepsilon _{1}^{x_{1},x_{2},x_{3}})$ $=$ $\varepsilon
		_{1}^{x_{1},x_{2},x_{3}}+\varepsilon _{0}^{x_{1},x_{2},x_{3}}\varepsilon
		_{1}^{y}$ $=$ $\varepsilon _{1}^{x_{1},x_{2},x_{3},y}$
		
		$\alpha (\varepsilon _{2}^{x_{1},x_{2},x_{3}})$ $=$ $\varepsilon
		_{2}^{x_{1},x_{2},x_{3}}+\varepsilon _{1}^{x_{1},x_{2},x_{3}}\varepsilon
		_{1}^{y}$ $=$ $\varepsilon _{2}^{x_{1},x_{2},x_{3},y},$
	\end{center}
	
	and
	
	\begin{center}
		$\alpha (\varepsilon _{3}^{x_{1},x_{2},x_{3}})$ $=$ $\varepsilon
		_{3}^{x_{1},x_{2},x_{3}}+\varepsilon _{2}^{x_{1},x_{2},x_{3}}\varepsilon
		_{1}^{y}$ $=$ $\varepsilon _{3}^{x_{1},x_{2},x_{3},y},$
	\end{center}
	
	inducing
	
	\begin{center}
		$\alpha (\mathcal{E}_{x_{1},x_{2},x_{3}})$ $\sqcup $ $\{\varepsilon
		_{4}^{x_{1},x_{2},x_{3},y}$ $=$ $x_{1}x_{2}x_{3}y\}$ $=$ $\mathcal{E}%
		_{x_{1},x_{2},x_{3},y}.$
	\end{center}
	
	So,
	
	\begin{center}
		$
		\begin{array}{ll}
		\mathcal{S}_{x_{1},x_{2},x_{3},y} & =\mathbb{C}\left[ \alpha (\mathcal{E}%
		_{x_{1},x_{2},x_{3}})\sqcup \{x_{1}x_{2}x_{3}y\}\right] \\ 
		&  \\ 
		& =\Phi \left( \mathcal{S}_{x_{1},x_{2},x_{3}}\right) \oplus \mathbb{C}%
		[x_{1}x_{2}x_{3}y].
		\end{array}
		$
	\end{center}
	
	\strut
\end{example}	

\section{\strut Applications}
In this section, we consider applications of (3.3.7) based on (2.2).

\subsection{Zeroes of Single-Variable Polynomials}

Let $f(z)$ be a degree-$n$ single-variable $\mathbb{C}$-polynomial, i.e., $f$ $%
\in $ $\mathcal{A}_{z}$ $=$ $\mathbb{C}[z].$ By the fundamental theorem of
algebra, $f(z)$ has its zeroes $\lambda _{1},$ ..., $\lambda _{n}$ (without
considering multiplicities), i.e.,

\begin{center}
	$f(z)$ $=$ $\overset{n}{\underset{j=1}{\Pi }}\left( z-\lambda _{j}\right) .$
\end{center}

For convenience, let

\begin{center}
	$x_{j}$ $=$ $-\lambda _{j},$ for $j$ $=$ $1,$ ..., $n.$
\end{center}

Then

(4.1.1)

\begin{center}
	$f(z)$ $=$ $\overset{n}{\underset{j=1}{\Pi }}\left( z+x_{j}\right) $ $=$ $%
	\sum_{k=0}^{n}\varepsilon _{k}^{x_{1},...,x_{n}}$ $z^{n-k}$
\end{center}

in $\mathcal{A}_{z},$ where $\varepsilon _{0}^{x_{1},...,x_{n}}$ $=$ $1,$ and

(4.1.2)

\begin{center}
	$
	\begin{array}{ll}
	\varepsilon _{k}^{x_{1},...,x_{n}} & =\varepsilon
	_{k}(x_{1},...,x_{n})=(-1)^{k}\varepsilon _{k}\left( \lambda _{1},\text{
		..., }\lambda _{n}\right) \\ 
	& =(-1)^{k}\varepsilon _{k}^{\lambda _{1},...,\lambda _{n}},
	\end{array}
	$
\end{center}

for $k$ $=$ $1,$ ..., $n.$

Now, suppose $f(z)$ is arbitrarily given (without knowing zeroes of $f(z)$)
in $\mathcal{A}_{z}$. Then one can construct the polynomial $f_{\lambda
	_{n+1}}(z)$ whose zeroes are the zeroes of $f(z)$ and $\lambda _{n+1}$ in $%
\mathbb{C}.$

\begin{proposition}
	Let $f(z)$ be a degree-$n$ polynomial $\sum_{k=0}^{n}t_{k}z^{n-k}$ in $%
	\mathcal{A}_{z},$ and let $\lambda _{n+1}$ $\in $ $\mathbb{C}.$ Then there
	exists a polynomial
	
	\begin{center}
		$f_{\lambda _{n+1}}(z)$ $=$ $\sum_{k=0}^{n+1}s_{k}$ $z^{(n+1)-k}$ $\in $ $%
		\mathcal{A}_{z}$
	\end{center}
	
	whose zeroes are the zeroes of $f(z)$ and $\lambda _{n+1}$ in $\mathbb{C}.$ In
	particular,
	
	\begin{center}
		$s_{0}$ $=$ $1,$
		
		$s_{k}$ $=$ $t_{k}-t_{k-1}\lambda _{n+1},$
	\end{center}
	
	for all $k$ $=$ $1,$ ..., $n,$ with additional identity: $t_{-1}=0,$ and
	
	\begin{center}
		$s_{n+1}$ $=$ $-t_{n}\lambda _{n+1},$
	\end{center}
	
	in $\mathbb{C}.$
\end{proposition}

\begin{proof}
	\strut Assume $f(z)$ is a degree-$n$ polynomial in $\mathcal{A}_{z}$ and $%
	\lambda _{n+1}$ $\in $ $\mathbb{C}$ are arbitrarily fixed, and suppose
	
	\begin{center}
		$f(z)$ $=$ $\sum_{k=0}^{n}t_{k}$ $z^{n-k}$ in $\mathcal{A}_{z}.$
	\end{center}
	
	By the fundamental theorem of algebra, $f(z)$ has its $n$-zeroes $\lambda
	_{1},$ ..., $\lambda _{n}$ (without considering the multiplicities). If we
	let $x_{j}$ $=$ $-\lambda _{j}$ in $\mathbb{C},$ for $j$ $=$ $1,$ ..., $n,$ then
	
	(4.1.3)
	
	\begin{center}
		$f(z)$ $=$ $\sum_{k=0}^{n}$ $t_{k}$ $z^{n-k}$ with $t_{k}$ $=$ $\varepsilon
		_{k}^{x_{1},...,x_{n}},$
	\end{center}
	
	by (4.1.1) satisfying (4.1.2). So, one can construct
	
	(4.1.4)
	
	$\qquad \qquad \varepsilon _{0}^{x_{1},...,x_{n},x_{n+1}}$ $=$ $1,$
	
	$\qquad \qquad \varepsilon _{k}^{x_{1},...,x_{n},x_{n+1}}$ $=$ $\alpha
	\left( \varepsilon _{k}^{x_{1},...,x_{n}}\right) ,$ for $k$ $=$ $1,$ ..., $%
	n, $
	
	where $\alpha $ is in the sense of (3.3.3), and
	
	$\qquad \qquad \varepsilon _{n+1}^{x_{1},...,x_{n},x_{n+1}}$ $=$ $%
	\varepsilon _{n}^{x_{1},...,x_{n}}x_{n+1}$ $=$ $x_{1}x_{2}...x_{n}x_{n+1},$
	
	by (3.3.4) (and by (3.3.7)), for
	
	\begin{center}
		$x_{n+1}$ $=$ $-$ $\lambda _{n+1}$ in $\mathbb{C}.$
	\end{center}
	
	Then, by (4.1.3), the quantities (4.1.4) are
	
	(4.1.5)
	
	$\qquad \qquad s_{0}$ $=$ $1,$
	
	$\qquad \qquad s_{k}$ $=$ $\alpha (t_{k})$ $=$ $t_{k}+t_{k-1}x_{n+1},$ for $%
	k $ $=$ $1,$ ..., $n,$
	
	and
	
	$\qquad \qquad s_{n+1}$ $=$ $t_{n}x_{n+1},$
	
	where $x_{n+1}$ $=$ $-\lambda _{n+1}.$
	
	In other words, one can construct the degree-$(n+1)$ polynomial
	
	\begin{center}
		$f_{\lambda _{n+1}}(z)$ $=$ $\sum_{k=0}^{n+1}$ $s_{k}$ $z^{(n+1)-k},$
	\end{center}
	
	whose constant term $s_{0}$ and coefficients $s_{j}$ satisfy (4.1.5) in $%
	\mathcal{A}_{z},$ such that the zeroes of $f_{\lambda _{n+1}}(z)$ are the
	zeroes of $f(z)$ and a given $\mathbb{C}$-quantity $\lambda _{n+1}.$
\end{proof}

\strut The above proposition is illustrated in the following example.

\begin{example}
	Let $f(z)$ $=$ $z^{4}-2z^{2}+z+3$ in $\mathcal{A}_{z},$ and let $i$ $\in $ $%
	\mathbb{C}.$ One can let
	
	\begin{center}
		$t_{0}$ $=$ $1,$ $t_{1}=-2,$ $t_{2}$ $=$ $0,$ $t_{3}$ $=$ $1,$ and $t_{4}$ $%
		= $ $3,$
	\end{center}
	
	in $\mathbb{C}.$ Then, for the fixed $\mathbb{C}$-quantity $i$, we have
	
	$\qquad \qquad \qquad s_{0}$ $=$ $1,$
	
	$\qquad \qquad \qquad s_{1}$ $=$ $t_{1}+$ $t_{0}i$ $=$ $-2+i,$
	
	$\qquad \qquad \qquad s_{2}$ $=$ $t_{2}+t_{1}i$ $=$ $0+(-2)i$ $=$ $-2i,$
	
	$\qquad \qquad \qquad s_{3}$\strut $=$ $t_{3}+t_{2}i$ $=$ $1+0i$ $=$ $1,$
	
	$\qquad \qquad \qquad s_{4}$ $=$ $t_{4}+t_{3}i$ $=$ $3+i,$
	
	and
	
	$\qquad \qquad \qquad s_{5}$ $=$ $t_{4}i$ $=$ $3i,$
	
	in $\mathbb{C}$, inducing a new degree-4 polynomial $f_{i}(z)$
	
	\begin{center}
		$f_{i}(z)$ $=$ $\sum_{k=0}^{5}s_{k}z^{5-k}$ $=$ $%
		z^{5}+(-2+i)z^{4}-2iz^{3}+z^{2}+i$ $z$ $+$ $3i.$
	\end{center}
	
	in $\mathcal{A}_{z}.$ Then this new degree-4 polynomial
	
	\begin{center}
		$z^{5}+(-2+i)z^{4}-2iz^{3}+z^{2}+i\,z$ $+$ $3i$
	\end{center}
	
	has its zeroes $i$ and all zeroes of $f(z).$
\end{example}

\strut The above proposition allows us to construct a degree-$(n+1)$
polynomial $f_{\lambda }(z)$ whose zeroes are the zeroes of $f(z)$ and $%
\lambda ,$ even though we do not know the zeroes of a fixed degree-$n$
polynomial $f(z)$.

\subsection{\strut JAVA Algorithm for $\mathbb{C}$-Quantities from Elementary
	Symmetric Polynomials}

In this section, as an application of (3.3.7) and (2.2), we establish a
computer algorithm for finding quantities from elementary symmetric
polynomials. This computer algorithm is constructed by the program language,
JAVA version 8.

\strut \strut

\strut \strut

\textbf{JAVA }(v.8)\textbf{\ Program}: Computing $\mathbb{C}$-quantities from
Elementary Symmetric Polynomials.

\strut \strut

import java.util.Scanner;

public class ComplexRecursiveAlgo \{

public static void main(String[] args) \{

Scanner in = new Scanner(System.in);

//Prompt for number of generators, n.

System.out.print(''Enter number of generating variables: '');

int n = in.nextInt();

System.out.println();

//Prompt for the n known generators, vector X.

System.out.print(''Enter the values of the n generators as a white space
separating list: '');

long[][] X = new long[n][2];

for(int i=0; i \textless n; i++)\{

X[i][0] = in.nextInt();

X[i][1] = in.nextInt();

\}

System.out.println();

//Recursively solve for the values of the elementary symmetric functions and
store in 2D array

long[][][] epsilon = new long[n][][];

for(int i=0; i \textless n; i++)

epsilon[i] = new long[i+1][2];

epsilon[0][0] = X[0];

for(int i=1; i \textless n; i++)\{

for(int k=0; k \textless epsilon[i].length; k++)\{

if(k-1 \textless 0)

epsilon[i][k] = add(epsilon[i-1][k], X[i]); //Since epsilon[i-1][k] = 1 for
k \textless 0

else if(k \textgreater i-1)

epsilon[i][k] = multiply(epsilon[i-1][k-1],X[i]); //Since epsilon[i-1][k] =
0 for k \textgreater i-1

else

epsilon[i][k] = add(epsilon[i-1][k], multiply(epsilon[i-1][k-1],X[i]));

\}

\}

//Provide values of epsilon to user

System.out.print(''Enter the values of n and k for the desired iteration:
'');

int N = in.nextInt() - 1;

int K = in.nextInt() - 1;

System.out.println();

System.out.println(''epsilon['' + (N+1) + '']['' + (K+1) + ''] = ('' +
epsilon[N][K][0] + '','' + epsilon[N][K][1] + '')'');

\}

public static long[] add(long[] z1, long[] z2)\{

long[] w = new long[2];

w[0] = z1[0] + z2[0];

w[1] = z1[1] + z2[1];

return w;

\}

public static long[] multiply(long[] z1, long[] z2)\{

long[] w = new long[2];

w[0] = z1[0]*z2[0] - z1[1]*z2[1];

w[1] = z1[0]*z2[1] + z1[1]*z2[0];

return w;

\}

\}

return w;

\}

\}

\strut

\strut \strut

\strut \strut \strut

\section*{}~
   Mr. Ryan Golden is an undergraduate at Saint Ambrose University.  He has worked  under Dr. Ilwoo Cho for the past two years and attributes his decision to pursue a PhD in applied math largely to this partnership.  His particular interests include dynamical systems theory, stochastic processes, and statistical learning theory.
   
   Ph.D. Ilwoo Cho has been a faculty member of the department of mathematics and statistics at Saint Ambrose University since 2005.  His research interests include free probability, operator algebra and theory, combinatorics, and groupoid dynamical systems.

\subsection*{}
   Ryan Golden, 536 Carlsbad Tr., Roselle, IL 60172, USA
   
   goldenryanm@sau.edu

\subsection*{}
   Ilwoo Cho, Dept. of Mathematics, Saint Ambrose University, 518 w. Locust St., Davenport, IA 52803, USA
   
   choilwoo@sau.edu

\end{document}